\documentclass[12pt,reqno]{amsart}
\usepackage{amssymb}
\usepackage[usenames,dvipsnames]{color}
\usepackage{euscript}
\usepackage{multicol}
\usepackage{amssymb}
\usepackage{amsmath}
\usepackage{times} 
\usepackage{amssymb}
\usepackage{amsmath}
\usepackage{latexsym}
\usepackage{amsthm}
\usepackage{epsfig}
\usepackage{color}
\usepackage{comment}
\usepackage{amsfonts,amssymb,mathrsfs,amscd}
\usepackage{bm}

\usepackage{enumitem}

\newcommand{\be}[1]{\begin{equation}\label{#1}}
\newcommand{\ee}{\end{equation}}
\newtheorem{theorem}{Theorem}[section]
\newtheorem{corollary}{Corollary}

\newtheorem{remark}{Remark}

\newtheorem{proposition}{Proposition}
\newtheorem{conjecture}{Conjecture}

\numberwithin{equation}{section}

\title
[Lieb--Thirring inequalities on manifolds with  negative curvature]
 {Lieb--Thirring inequalities on manifolds with constant negative curvature}

\author{Alexei Ilyin}
\address{Alexei Ilyin:  Keldysh Institute of Applied Mathematics and
Institute for Information Transmission Problems;
ilyin@keldysh.ru}

\author{Ari Laptev}
\address{Ari Laptev:  Department of Mathematics,
Imperial College London,
 and  Sirius University of Science and Technology  Olimpiyskiy ave. b.1,
 Sirius, Krasnodar region, Russia, 354340;
a.laptev@imperial.ac.uk}

\author{Timon Weinmann}
\address{Timon Weinmann:  Department of Mathematics,
Imperial College London;
t.weinmann21@imperial.ac.uk}

\begin{document}

\maketitle

\begin{quote}
{\normalfont\fontsize{8}{10}
\selectfont{\bfseries  Abstract.}
In this short note we prove Lieb--Thirring inequalities
on manifolds with negative constant curvature. The discrete spectrum appears
below the continuous spectrum $(d-1)^2/4, \infty)$, where $d$ is the dimension
of the hyperbolic space.
As an application we obtain a P\'olya type inequality with not a sharp constant.
An example of a 2D domain is given for which numerical calculations suggest that
the P\'olya inequality holds for it.}
\end{quote}

\setcounter{equation}{0}
\section{Introduction}\label{Sec:1}

Lieb--Thirring inequalities have important applications in
mathematical physics, analysis, dynamical systems and attractors, to
mention a few. A current state of the art  of many aspects of the
theory is presented in~\cite{lthbook}. We mention here the
celebrated paper by Lieb and Thirring~\cite{LTh}, where such
inequalities were studied for the questions of stability of matter.

In certain applications Lieb--Thirring inequalities are  considered
on a manifold. For example, on  torus or sphere one has to
impose the zero mean orthogonality condition, see \cite{IMRN}, \cite{ILZ-JFA}, \cite{IZ},
\cite{I-L-MS}, \cite{UMN}. Such inequalities are useful in the study of the dimension of attractors in theory of Navier-Stokes equation.

In this work we prove Lieb--Thirring inequalities on manifolds with negative constant curvature.
Let $\Bbb H^d$, $d\ge2$, be the upper half-space
$$
\Bbb H^d = \{(x,y):\, x\in \Bbb R^{d-1},\, y\in \Bbb R_+\}
$$
with the Poincare metric $ds^2 = y^{-2}(dx^2+ dy^2)$.
We consider the self-adjoint Laplace--Beltrami operator in
$L^2\left(\Bbb H^d, \frac{dx\, dy}{y^d}\right)$
\begin{equation}\label{Delta_h}
-\Delta_h = - y^d \frac{\partial}{\partial y}  \, y^{2-d} \, \frac{\partial}{\partial y} - y^{2} \,
\sum_{n=1}^{d-1} \frac{\partial^2}{\partial x_n^2}.
\end{equation}
The spectrum of the standard Laplacian
$$
-\Delta=  - \sum_{n=1}^d \frac{\partial^2}{\partial x_n^2}
$$
acting in $L^2 (\Bbb R^d, dx)$ is absolutely continuous and covers the whole half-line $[0,\infty)$.
By contrast, the spectrum of the Laplace operator \eqref{Delta_h}
is continuous and covers the interval $[(d-1)^2/4, \infty)$, see for example  \cite{MCK}, \cite{Perry}\, \cite{LaxPhil}.

Denote by $L_{\gamma,d}^{\rm cl}$ the value
$$
L_{\gamma,d}^{\rm cl} = \frac{1}{(2\pi)^d}\, \int_{\Bbb R^d} (1-|\xi|^2)_+^\gamma \, d\xi =
\frac{\Gamma(\gamma+1)}{(4\pi)^{d/2} \Gamma(\gamma + d/2 +1)}.
$$
Let $V= V(x)\ge 0$. The classical Lieb--Thirring  inequality for a
Schr\"odinger operator $-\Delta-V$ on $L^2(\Bbb R^d)$ states
$$
{\rm Tr}\, \left(-\Delta- V\right)_-^{\gamma}\le L_{\gamma, d} \int_{\Bbb R^d} V^{\gamma+d/2} \, dx.
$$
In ~\cite{LTh} E.H. Lieb and W. Thirring proved that it holds
for finite constants $L_{\gamma, d}$ as long as $\gamma>\max(0,1-d/2)$.
In the case $d\ge3, \,\gamma=0$ this bound is known as the
Cwikel--Lieb--Rozenblum (CLR) inequality, see \cite{C,L,R}.
The second critical case $d=1, \gamma=1/2$ was settled by Weidl in \cite{Timo}.
The inequality is known to fail for $d=2, \gamma=0$. For a comprehensive
treatment and references of the subject, see \cite{lthbook}.

\medskip
In this paper we study the spectrum of the Schr\"odinger operator
\begin{equation}\label{Schr}
-\Delta_h - V,
\end{equation}
acting in $L^2\left(\Bbb H^d, \frac{dx\, dy}{y^d}\right)$ and
obtain Lieb--Thirring inequalities for the discrete spectrum below
$(d-1)^2/4$.
It is convenient to denote the eigenvalues
$\{\lambda_k\}$ of the operator \eqref{Schr} in terms of negative
values $\{-\mu_k\}$, where
\begin{equation}\label{mu_subst}
\lambda_k = \frac{(d-1)^2}{4} -\mu_k.
\end{equation}
Let us denote by $R_{1,1}$ the value $R_{1,1}\le 1.456\dots $ that was obtained
in the recent paper  \cite{Frank-Nam} on the Lieb--Thirring inequality
for a one-dimensional Schr\"odinger operator with an operator-valued potential.

\medskip
The main result of the paper is the following:
\begin{theorem}\label{main}
Let $V\ge 0$ and $\gamma\ge 1/2$. Then
$$
\sum \mu_k^\gamma \le L_{\gamma, d} \, \int_{\Bbb H^d} V(x,y)^{\gamma+d/2} \,  \frac{dx\, dy}{y^d},
$$
where
\begin{equation}\label{Lgamma}
L_{\gamma,d}=
\begin{cases}
L_{\gamma,d}^{\rm cl}, & \gamma \ge 3/2,\\
R_{1,1} L_{\gamma,d}^{\rm cl} & 1 \le \gamma< 3/2, \\
2 R_{1,1} L_{\gamma,d}^{\rm cl} & 1/2 \le \gamma< 1.
\end{cases}
\end{equation}
\end{theorem}

In the next Section~\ref{Sec:2} we obtain a simple proof of the fact that the
continuous spectrum of the Laplacian on hyperbolic space with curvature $-1$
covers the semi-axis $[(d-1)^2/4,\infty)$. In Section~\ref{Sec:3} we give the
proof the main Theorem \ref{main} and in Section~\ref{Sec:4} we obtain the dual
inequality that could be used for estimates of the dimension of attractors in
theory of the Navier--Stokes equation.

In Section~\ref{Sec:5} we apply Theorem \ref{main} to derive an inequality on
the number of eigenvalues below $\Lambda>0$ for Dirichlet Laplace--Beltrami
operator on a domain $\Omega\subset\Bbb H^d$ of finite measure.

Assume that $\Omega \subset{\overline\Omega} \subset\Bbb H^d$ satisfies the inequality
\begin{equation*}
|\Omega|_h = \int_\Omega \frac{dxdy}{y^d} <\infty.
\end{equation*}
We consider the Dirichlet eigenvalue problem for the Laplace--Beltrami operator
$-\Delta_h$ in $L^2(\Omega, y^{-d} dx dy)$
\begin{equation}\label{dirichlet}
-\Delta_h u = \lambda u, \qquad u\big|_{(x,y)\in \partial\Omega} = 0.
\end{equation}
The spectrum of this operator is discrete and we denote by $\{\lambda_k\}$  its eigenvalues.
Such eigenvalues satisfy the inequality
$$
\lambda_k > \frac{(d-1)^2}{4}.
$$
Similarly to \eqref{mu_subst} it is convenient to introduce number $\nu_k $ such that
$$
\lambda_k = \frac{(d-1)^2}{4} +  \nu_k
$$
and study the counting function $\mathcal N(\Lambda)$  of the spectrum
$$
\mathcal N(\Lambda) = \#\{k: \, \nu_k <\Lambda\}, \qquad \Lambda>0.
$$
\begin{theorem}\label{NL}
Let $|\Omega|_h<\infty$. Then the counting function $\mathcal N(\Lambda)$ of the eigenvalues of the spectral problem \eqref{dirichlet} satisfies the following inequality
\begin{equation}\label{N}
\mathcal N(\Lambda) \le \left(1+\frac{2}{d}\right)^{d/2} \,
\left(1+\frac d2\right)L_{1,d} \, \Lambda^{d/2} |\Omega|_h,
\end{equation}
where $L_{1,d}$ is the constant from Theorem \ref{main},
so that
$$
(1+d/2)L_{1,d}\le R_{1,1}(1+d/2)\,L_{1,d}^{\rm cl}=
R_{1,1}\,L_{0,d}^{\rm cl}.
$$
\end{theorem}

The inequality \eqref{N} is a P\'olya type inequality \cite{P} for manifolds with constant negative curvature, where, we believe, the constant is not sharp.

\medskip
\noindent
\begin{conjecture}\label{conj} For the counting function $\mathcal N(\Lambda)$  of the eigenvalues $\lambda_k = (d-1)^2/4 + \nu_k$ of the Dirichlet boundary value problem \eqref{dirichlet}
we have
\begin{equation}\label{polya}
\mathcal N(\Lambda) \le L_{0,d}^{\rm cl} \, \Lambda^{d/2}\, |\Omega|_h.
\end{equation}
\end{conjecture}

\begin{remark}
At the moment we do not have any examples of $\Omega$ for which the inequality \eqref{polya} holds.
\end{remark}

In Section~\ref{Sec:6} we consider the special case of Theorem \ref{NL}, where $\Omega = \widetilde\Omega \times (a,b)$, $\widetilde\Omega\subset \Bbb R^{d-1}$, is a domain of finite Lebesgue measure and $0<a<b\le\infty$. This additional structure of $\Omega$ allows us to obtain a better constant than the constant that could be derived from Theorem \ref{main} in the case $1/2\le \gamma < 1$, see Theorem \ref{product}. Unfortunately it does not imply the improvement of the constant found in Theorem \ref{NL}.

Finally in Section~\ref{Sec:7}  we give an example
of a domain in $\mathbb H^2$ supporting the conjecture \ref{conj} using numerics.


\setcounter{equation}{0}
\section{Some preliminary results}\label{Sec:2}

In \cite{MCK} the author gives a simple proof of the fact that the continuous spectrum of the operator $-\Delta_h$ in dimension two coincides with the interval $[1/4,\infty)$ by using the Cauchy-Schwarz inequality. Besides, he gives a more complicated proof of the fact that  in the case of $\Bbb H^d$, $d>2$ the continuous spectrum fills the half-axis $[(d-1)^2/4, \infty)$.

In this section we present a simple proof of the following well-known fact.
\begin{proposition}
Let $-\Delta_h$ be the Laplacian in $L^2\left(\Bbb H^d, \frac{dx\, dy}{y^d}\right)$. Then the continuous spectrum coincides with $\sigma_c = [(d-1)^2/4,\infty)$.
\end{proposition}

\begin{proof}
Let us consider the quadratic form of the operator \eqref{Delta_h}
$$
(-\Delta_h u, u) = \int_{\Bbb H^n}  y^{2-d} \,\left(|\partial_y u|^2
+\sum_{n=1}^{d-1} |\partial_{x_n} u|^2\right)  \, dx \, dy.
$$
The substitution
\begin{equation}\label{subst}
y = e^t, \qquad u = e^{\frac{d-1}{2} \,t} \, v,
\end{equation}
implies
$$
\iint_{\Bbb H^d} |u|^2 \, \frac{dx\, dy}{y^d} =
\int_{\Bbb R^n} |v|^2\, dx\, dt.
$$
and
$$
(-\Delta_h u, u) = \int_{\Bbb R^d}\left( |v_t'|^2 + \frac{(d-1)^2}{4}\, |v|^2  + e^{2t} \sum_{n=1}^{d-1}
|\partial_{x_n} v|^2\right) \, dxdt.
$$
Thus we reduce the hyperbolic Laplacian to the operator in $L^2(\Bbb R^d) $
$$
-\frac{\partial^2}{\partial t^2} - e^{2t} \, \Delta_x  + \frac{(d-1)^2}{4}.
$$
Obviously the spectrum of the differential part of the above expression coincided with $[0,\infty)$ and the term $\frac{(d-1)^2}{4}$ gives the required shift of the spectrum.
The proof is complete.
\end{proof}

\begin{remark} It might be interesting to obtain a simple proof of properties of the spectrum of the operator $-\Delta_h$ in the case when the negative curvature is not a constant.
\end{remark}

\medskip
In order to prove Theorem \ref{main} we need to recall some results on 1D Schr\"odinger operators with operator-valued potentials.

\begin{proposition}\label{2}
Let $Q=Q(x)\ge 0$ be self-adjoint operator-valued function in a Hilbert space $G$ for a.e. $x\in\Bbb R$. We assume that ${\rm Tr}\, Q(\cdot) \in L^{\gamma+1/2} (\Bbb R, G)$, $\gamma\ge1/2$.
Then
$$
{\rm Tr}\, \left(\frac{d^2}{dx^2} \otimes I_{G} - Q\right)_-^\gamma \le
L_{\gamma,1} \, \int_{\Bbb R} {\rm Tr}\, Q^{\gamma+1/2}\, dx,
$$
where $I_{G}$ is the identity operator in $G$.
\end{proposition}

If $\gamma = 1/2$ then the constant $L_{\gamma=1/2,1} = 2L_{1/2,1}^{\rm cl}$ and it is sharp. This was obtained in \cite{HLW} and from this one immediately obtains that
$L_{\gamma,1}\le 2L_{\gamma,1}^{\rm cl}$ with $1/2\le \gamma<1$. (For the scalar case the sharp constant in the case $\gamma = 1/2$ was obtained in \cite{HLThom}.

If $\gamma=1$ then the sharp constant is unknown and the best known constant
for some years was refereed to  \cite{DLL} (see also~\cite{EF}). It is only recently this constant
was improved in the paper \cite{Frank-Nam}, where the authors found that
$L_{1,1} \le R_{1,1} L_{1,1}^{\rm cl}$ with
$R_{1,1}\le1.456 \dots.$  This leads to the estimate
$L_{\gamma,1}\le R_{1,1} L_{\gamma,1}^{\rm cl}$ with $1\le \gamma<3/2$.

Finally for any $\gamma\ge3/2$ the above Proposition was proved in \cite{Lap-Weid} and in this case we have sharp constants $L_{\gamma,1} = L_{\gamma,1}^{\rm cl}$.

In all cases the authors first obtained inequalities for $\gamma=1/2, 1$ and $3/2$ that were afterwards extended to arbitrary $\gamma$'s using lifting argument found by Aizenman and Lieb \cite{AizLieb}.

\setcounter{equation}{0}
\section{The proof on the main result}\label{Sec:3}

Let us consider the quadratic form of the operator \eqref{Delta_h}
$$
(-\Delta_h u, u) = \int_{\Bbb H^n}  y^{2-d} \,\left(|\partial_y u|^2
+\sum_{n=1}^{d-1} |\partial_{x_n} u|^2\right)  \, dx \, dy.
$$
Applying the exponential change of variables that was introduced
in the previous section we reduce the problem to the study of
the spectrum defined by the form in $L^2(\Bbb R^d)$
$$
\int_{\Bbb R^d}\left( |v_t'|^2   + e^{2t} \sum_{n=1}^{d-1}
|\partial_{x_n} v|^2    - V(x,e^t) |v|^2 \right) \, dxdt = -\mu \, \int_{\Bbb R^d} |v|^2 \, dx dt.
$$

 \medskip
 \noindent
Using the variational principle and the Lieb--Thirring inequalities for 1D
Schr\"odinger operators with operator-valued symbols, see Proposition \ref{2},
we obtain
\begin{multline*}
\sum \mu_k^\gamma(-\Delta_h -V) \le \sum \mu_k^\gamma\left(-\frac{d^2}{dt^2} -
\left(e^{2t} \, \sum_{n=1}^{d-1}
\frac{\partial^2}{\partial x_n^2} + V(x,e^{t})\right)_+ \right)\\
 \le L_{\gamma,1}  \, \int_{\Bbb R} {\rm Tr}\, \left( -e^{2t}\, \sum_{n=1}^{d-1}
\frac{\partial^2}{\partial x_n^2} - V(x,e^{t})\right)^{\gamma +1/2}_- \, dt.
\end{multline*}
This trick reduced the problem to a Schr\"odinger operator in $L^2(\Bbb R^{d-1})$,
where the exponential $e^{2t}$ is just a parameter:
\begin{multline*}
{\rm Tr}\, \left( -e^{2t}\, \sum_{n=1}^{d-1}
\frac{\partial^2}{\partial x_n^2} - V(x,e^{t})\right)^{\gamma +1/2}_-\\=
e^{2t(\gamma+1/2)}{\rm Tr}\, \left( -\sum_{n=1}^{d-1}
\frac{\partial^2}{\partial x_n^2} -e^{-2t}\,  V(x,e^{t})\right)^{\gamma +1/2}_-.
\end{multline*}
Therefore, applying the standard Lieb-Thirring inequality in dimension $d-1$, we find
\begin{multline*}
\sum \mu_k^\gamma \le
L_{\gamma,1} L_{\gamma+1/2, d-1} \, \int_{\Bbb R^d}
e^{(1-d)t}\,V(x,e^{t})^{\gamma +d/2} \, dxdt \\
=  L_{\gamma, d} \, \int_{\Bbb H^d}
V(x,y)^{\gamma +d/2} \, \frac{dx dy}{y^d}.
\end{multline*}
The best known constants $L_{\gamma,d}$ in the latter equality
are defined by the best constants that appear in Proposition \ref{2}.
The relation
$L_{\gamma,1} L_{\gamma+1/2, d-1}=L_{\gamma, d}$
follows from the  relation
$L_{\gamma,1}^{\rm cl} L_{\gamma+1/2, d-1}^{\rm cl}=L_{\gamma, d}^{\rm cl}$
and definition~\eqref{Lgamma}.
The proof is complete.


\setcounter{equation}{0}
\section{Dual inequalities}\label{Sec:4}

Consider an orthonormal set of function $\{u_m\}_{m=1}^M$ in $L^2(\Bbb H^d, y^{-d} \, dx dy)$. Then using the exponential change
$$
y=e^t, \qquad u_m = e^{(d-1)t/2}\, v_m
$$
we find
$$
\delta_{m,l} = \int_{\Bbb H^d} u_m {\overline u_l} \frac{dx dy}{y^d} =
\int_{\Bbb R^d} v_m {\overline v_l} \, dx dt.
$$
Namely, this shows that if the functions $\{u_m\}_{m=1}^M$ are orthonormal in $L^2(\Bbb H^d, y^{-d} \, dx dy)$ then the functions $\{v_m\}_{m=1}^M$ are orthonormal in
$L^2(\Bbb R^d)$.

\medskip
Assuming $\gamma = 1$ we obtain
\begin{multline*}
\sum_{m=1}^M
\int_{\Bbb R^d} \left(|\partial_t v_m|^2 + e^{2t}\,\sum_{n=1}^{d-1} |\partial_{x_n} v_m|^2  - V(x,e^t) |v_m|^2 \right) \,dx dt\\
\ge   - \sum_m \mu_m
\ge
- L_{1, d} \, \int_{\Bbb R^d} V(x,e^t)^{1 +d/2} \, dxdt.
\end{multline*}

\noindent
Thus
\begin{multline*}
\sum_{m=1}^M \int_{\Bbb R^d}\left(|\partial_t v_m|^2 +
e^{2t}\, \sum_{n=1}^{d-1} |\partial_{x_n} v_m|^2 \right) \,dx dt \\
\ge
\int_{\Bbb R^d} \left(V(x,e^t)\, \sum_{m=1}^M  |v_m|^2  - L_{1, d} \,
V(x,e^t)^{1 +d/2}\right) \, dx dt \\
= \int_{\Bbb R^d} \left(V(x,e^t)\, \widetilde\rho  - L_{1, d} \,
V(x,e^t)^{1 +d/2}\right) \, dx dt,
\end{multline*}
where $\widetilde\rho = \sum_{m=1}^N  |v_m|^2$. We now choose
$$
V = \left( \frac{\widetilde\rho}{L_{1,d} (1 + d/2)} \right)^{2/d}
$$
and find
\begin{equation}\label{tilderho}
\sum_{m=1}^M
\int_{\Bbb R^d}\left(|\partial_t v_m|^2 +
e^{2t}\sum_{n=1}^{d-1} |\partial_{x_n} v_m|^2 \right) \,dx dt
\ge
K_{1,d} \, \int_{\Bbb R^d} (\widetilde\rho)^{1+ 2/d}  \, dxdt,
\end{equation}
where
$$
K_{1,d} = \frac{2}{d} \left(1 + \frac{d}{2}\right)^{1+2/d} L_{1,d}^{2/d}.
$$
Returning to the orthonormal system of functions $\{u_m\}$ and denoting
by $\rho = \sum_{m=1}^M |u_m|^2$ we obtain
$$
\int_{\Bbb R^d} (\widetilde\rho)^{1+ 2/d}  \, dxdt = \int_{\Bbb H^d} y^{\frac{2(1-d)}{d}}
\rho^{1+2/d}\, \frac{dxdy}{y^d}.
$$
Besides, when passing from the quadratic forms in the left hand side of
\eqref{tilderho}  to the quadratic forms  $(-\Delta_h u_m,u_m)$ we
have to add the shift
$$
(d-1)^2/4 \|u_m\|_{L^2(\Bbb H^d,y^{-d}dxdy)} = (d-1)^2/4, \quad m=1,2,\dots, M.
$$
Finally we have

\bigskip
\noindent
\begin{theorem}\label{dual}
Let $\gamma=1$ and let $\{u_m\}_{m=1}^M$ be an orthonormal system of function in $L^2(\Bbb H^d, y^{-d} dx dy)$. Then
$$
\sum_{m=1}^M (-\Delta_h u_m,u_m) \ge K_{1,d} \int_{\Bbb H^d} y^{\frac{2(1-d)}{d}}
\rho^{1+2/d}\, \frac{dxdy}{y^d}
+ M\, \frac{(d-1)^2}{4}.
$$
\end{theorem}

Assume that $M=1$ and $u\in H^1(\Bbb H^d, y^{-d} dxdy)$ and denote
$$
\|u\| = \left( \int_{\Bbb H^d} |u|^2\, \frac{dxdy}{y^d}\right)^{1/2}.
$$
Then
\begin{corollary}\label{sobolev}
For a function $u\in H^1(\Bbb H^d, y^{-d} dxdy)$ we have
\begin{multline*}
\|u\|^{4/d} \, \int_{\Bbb H^d} y^{2-d} \left(|\partial_y u|^2 + |\nabla_x u|^2\right) \, dx dy  \\
\ge K_{1,d} \int_{\Bbb H^d} y^{\frac{2(1-d)}{d}}  |u|^{2 + 4/d} \,\,\frac{dx dy}{y^d} + \frac{(d-1)^2}{4}  \, \|u\|^{2+4/d}.
\end{multline*}
\end{corollary}


\setcounter{equation}{0}
\section{Proof of Theorem \ref{NL}}\label{Sec:5}

Let $\Omega \subset{\overline\Omega} \subset\Bbb H^d$ be a domain of finite measure $|\Omega|_h$ and let us introduce the following potential
$$
V =
\begin{cases}
1, & (x,y)\in \Omega,\\
-\infty, & (x,y) \not\in\Omega.
\end{cases}
$$
Then the problem of the study of the counting function $\mathcal N(\Lambda)$ is reduced to the study of the number of the negative eigenvalues of the operator in $L^2(\Bbb H^d, y^{-d} dxdy)$.
$$
-\Delta_h - (d-1)^2/4 - \Lambda V,
$$
because the above operator is unitary equivalent to the operator
$$
-\Delta_h - (d-1)^2/4 - \Lambda
$$
in $L^2(\Omega, y^{-d} dxdy)$ with Dirichlet boundary conditions.
This number  coincides with the number of $k$'s such that $\nu_k<\Lambda$. Therefore
using the variational principle and applying Theorem \ref{main} with $\gamma = 1$ we find
$$
\sum_k(\nu_k -\Lambda)_-  = \sum_k(\Lambda -\nu_k)_+
\le L_{1,d} \, \Lambda^{1+d/2} \, \int_{\Omega} \frac{dxdy}{y^d} =
L_{1,d} \, \Lambda^{1+d/2} \, |\Omega|_h.
$$
Then for any $\Upsilon>\Lambda$
$$
\mathcal N(\Lambda) \le \frac{1}{\Upsilon - \Lambda}\, \sum_k (\Upsilon -\nu_k)_+
\le L_{1,d} \, \frac{\Upsilon^{1+d/2}}{\Upsilon - \Lambda} \, |\Omega|_h.
$$
Minimising the right hand side of the latter inequality we find
$$
\Upsilon = \Lambda \frac{1+d/2}{d/2}
$$
and thus
\begin{equation}\label{NLambda}
\mathcal N(\Lambda) \le
\left(1+\frac d2\right)
\left(1+\frac{2}{d}\right)^{d/2} \, R_{1,1} \, L_{1,d}^{\rm cl} \, |\Omega|_h\, \Lambda^{d/2}.
\end{equation}
This concludes the proof of Theorem \ref{NL}.

\setcounter{equation}{0}
\section{A special case of Theorem \ref{NL}}\label{Sec:6}

\medskip
\noindent
Let  $\Omega= \widetilde\Omega \times (a,b)\subset\Bbb H^d$, where $\widetilde\Omega\subset \Bbb R^{d-1}$ and $0<a<b\le\infty$.
We assume that the Lebesgue measure $|\widetilde\Omega|<\infty$ and consider the Dirichlet problem in  $L^2\left(\Omega, y^{-d} dxdy\right)$
\begin{equation*}
-\Delta_h u = \lambda u, \qquad u\big|_{\partial\Omega} = 0,
\end{equation*}
Using the substitution \eqref{subst} we reduce the problem to
\begin{equation}\label{reduced_lapl}
\left(-\partial_t^2 - e^{2t} \Delta_x  + \frac{(d-1)^2}{4}\right) \,v(x,t)  = \lambda v(x,t) \qquad  v\big|_{\widetilde\Omega \times (\alpha,\beta)} = 0,
\end{equation}
where $\alpha= \ln a$, $\beta = \ln b$. As before
it is convenient to introduce values $\nu$
$$
\lambda = \frac{(d-1)^2}{4} + \nu
$$
Due to the product structure of $\Omega$ the eigenfunctions $\{v_{\ell k}\}_{\ell, k=1}^\infty$ of the problem can be found as the product
$$
v_{\ell k}(x,t) = \varphi_\ell(x) \psi_{\ell k}(t),
$$
where $\varphi_\ell$ satisfy the Diriclet boundary value problem
$$
-\Delta \varphi_\ell = \varkappa_\ell \, \varphi_\ell, \qquad \varphi_\ell\big|_{\partial\widetilde\Omega} = 0.
$$
and
$$
-\partial_t^2\, \psi_{\ell k}(t) + e^{2t}\, \varkappa_{\ell} \, \psi_{\ell k} (t) = \nu_{\ell k} \, \psi_{\ell k}(t), \qquad
\psi_{\ell k}(t)\big|_{t=\alpha,\beta} = 0.
$$
Note that the functions $ \{\varphi_\ell\}_{\ell=1}^\infty $ give us an orthonormal basis in $L^2(\widetilde\Omega)$ and for each fixed $\ell$ the  set of functions $\{\psi_{\ell k}\}_{k=1}^\infty$ is an orthonormal basis in $L^2(\alpha,\beta)$.

\medskip
\noindent
Altogether we have the equation
\begin{equation*}
\left(-\partial_t^2 - e^{2t} \Delta_x\right) \, \varphi_\ell(x)\, \psi_{\ell k}(t) =  \nu_{\ell k}  \,  \varphi_\ell(x)\,  \psi_{\ell k}(t),
\end{equation*}
where
$$
\left\{ \frac{(d-1)^2}{4} + \nu_{\ell k}\right\}_{\ell,k=1}^\infty
$$
are all eigenvalues of the problem \eqref{reduced_lapl}.

\noindent
Using the notations from Section \ref{Sec:5} and applying
Lieb--Thirring inequalities for the $1/2$-moment of the
Schr\"odinger operator with the operator-valued potential
$-e^{2t}\Delta_x - \Lambda$ with the sharp constant
$ 2\, L_{1/2,1}^{\rm cl}$ (see \cite{HLW}) we obtain
\begin{multline*}
\sum_{\ell,k=1}^\infty (\Lambda - \nu_{\ell k})_+^{1/2} =
{\rm Tr}\, \left(-\partial^2_t - e^{2t}\, \Delta_x -\Lambda\right)_-^{1/2} \\
\le
2\, L_{1/2,1}^{\rm cl}
\int_\alpha^\beta {\rm Tr}\, \left(\Lambda +  e^{2t}\Delta_x \right)_+  dt.
\end{multline*}
The multiplier $e^{2t}$ in the study of the trace  ${\rm Tr}\, \left(\Lambda +  e^{2t}\Delta_x \right)_+$ could be considered as a constant.
Therefore applying Berezin--Li $\&$Yau inequality for the Dirichlet Laplacian in $\widetilde\Omega$  (see \cite{B} \cite{LYa} and also \cite{Lap})  we find
\begin{multline*}
{\rm Tr}\,
 \sum_{\ell} (\Lambda-  e^{2t}\varkappa_{\ell})_+  \\
\le
 |\widetilde\Omega| \, (2\pi)^{1-d} \, \int_{\Bbb R^{d-1}} (\Lambda - e^{2t}\,|\xi|^2)_+ \, d\xi = L_{1,d-1}^{\rm cl}\, |\widetilde\Omega| \,\Lambda^{(d+1)/2} \,e^{(1-d)t}.
\end{multline*}
Finally by using
$$
L_{1/2,1}^{\rm cl} L_{1,d-1}^{\rm cl} = L_{1/2,d}^{\rm cl}
$$
and  returning to variables $(t,x)$ we arrive at
\begin{multline*}
\sum_{\ell,k=1}^\infty (\Lambda - \nu_{\ell k})_+^{1/2} \le
\Lambda^{\frac{d+1}{2}} \,2\, L_{1/2,d}^{\rm cl}
\int_\alpha^\beta e^{(1-d)t} \, dt \, |\widetilde\Omega| \\
= \Lambda^{\frac{d+1}{2}} \,2\, L_{1/2,d}^{\rm cl} \int_{\Omega}
\frac{dx dy}{y^d} =\Lambda^{\frac{d+1}{2}} \, 2\, L_{1/2,d}^{\rm cl} \,
|\Omega|_h.
\end{multline*}
Using the standard Aizenman--Lieb arguments we can extent the above inequality to the $\gamma$-Riesz means with $1/2\le\gamma<1$ and obtain

\begin{theorem}\label{product}
Let  $\Omega= \widetilde\Omega \times (a,b)\subset\Bbb H^d$, where $\widetilde\Omega\subset \Bbb R^{d-1}$ and $0<a<b\le\infty$. Then for the values
$\{\nu_{\ell k}\}_{\ell,k=1}^\infty $ related to the eigenvalues of the Dirichlet Laplacian  \eqref{Delta_h} via the equation
$
\lambda_{\ell k} = \frac{(d-1)^2}{4} + \nu_{\ell k}
$
we have
\begin{equation}\label{1/2Special}
\sum_{\ell,k=1}^\infty (\Lambda - \nu_{\ell k})_+^{\gamma} \le
\Lambda^{\frac{d}{2}+\gamma} \, 2\, L_{\gamma,d}^{\rm cl} \,|\Omega|_h, \quad 1/2\le\gamma<1.
\end{equation}
\end{theorem}

\begin{remark}
The constant $2\, L_{\gamma,d}^{\rm cl}$, $1/2\le\gamma<1$ is better than the constant $2\, R_{1,1}\, L_{\gamma,d}^{\rm cl}$ that could be obtained from Theorem \ref{main}.
\end{remark}

Similarly to Section \ref{Sec:5} we can use the inequality \eqref{1/2Special} for estimating the counting function
$\mathcal N(\Lambda)$ for spectrum of the Dirichlet Laplacian \eqref{Delta_h} in domain with the product structure. Indeed, for any $\Upsilon>\Lambda$
$$
\mathcal N(\Lambda) \le \frac{1}{(\Upsilon - \Lambda)^{1/2}}\,
\sum_{\ell,k =1}^\infty(\Upsilon -\nu_{\ell k})_+^{1/2}
\le 2L_{1/2,d}^{\rm cl} \, \frac{\Upsilon^{(d+1)/2}}{(\Upsilon - \Lambda)^{1/2}} \, |\Omega|_h.
$$
Minimising the right hand side of the latter inequality we find
$
\Upsilon = \Lambda \frac{1+d}{d}
$
and thus
\begin{equation}\label{NSpecial}
\mathcal N(\Lambda) \le
\left(\frac{d+1}{d}\right)^{(d+1)/2} \, \sqrt{d}\, 2\, L_{1/2,d}^{\rm cl} \, |\Omega|_h\, \Lambda^{d/2}.
\end{equation}
However, the ratio of the constants \eqref{NSpecial} and \eqref{NLambda}
is greater than one and therefore Theorem \ref{product} does not imply
any improvement for the inequality \eqref{NLambda}.

\begin{figure}[htb]
	\centerline{\psfig{file=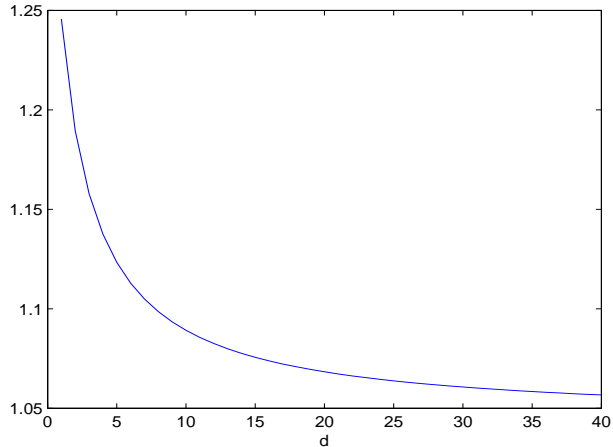,width=9.5cm,height=6.5cm,angle=0}}
	\caption{The graph of the ratio of the constants in
	\eqref{NSpecial} and \eqref{NLambda}: \eqref{NSpecial}$(d)$/\eqref{NLambda}$(d)$.}
	\label{Fig.1}
\end{figure}

\bigskip

\setcounter{equation}{0}
\section{A numerical example supporting conjecture \ref{conj} }\label{Sec:7}

Let $d=2$ and let
$$
\Omega =(0,\pi) \times (e^{-1},e) = \left\{(x,y):\, x\in(0,\pi),\,  y\in(e^{-1},e)\right\}
$$
with
$$
|\Omega|_h=\pi(e-e^{-1}).
$$
 Consider the Dirichlet Laplacian $-\Delta_h$ in
$L^2(\Omega, y^{-2} dxdy)$ defined in \eqref{Delta_h}. Using the
notations from Section \ref{Sec:6} we have $\alpha=-1$, $\beta=1$.
Obviuosly the eigenvalues of the problem
$$
-\partial^2_x \varphi= \varkappa \varphi, \quad \varphi\big|_{x=0,\pi} = 0
$$
are $ \varkappa_\ell = \ell^2$.

The arguments from Section \ref{Sec:6} imply that the problem is reduced
to the study of the eigenvalues  $\nu_{\ell k}$ satisfying the equation
\begin{equation}\label{lk}
-\partial_t^2\, \psi_{\ell k}(t) + e^{2t}\, \ell^2 \, \psi_{\ell k} (t) = \nu_{\ell k} \, \psi_{\ell k}(t), \qquad
\psi_{\ell k}(t)\big|_{t=-1,1} = 0.
\end{equation}

 Now comes numerics  to prove the following  P\'olya type inequality:

\begin{multline*}
\mathcal N(\Lambda) = \# \{\ell,k: \, \nu_{\ell k}< \Lambda\}
\le
(2\pi)^{-2} \int_0^\pi \int_{-1}^1 \int_{\xi_2^2 + e^{2t} \xi_1^2<\Lambda} d\xi_1 d\xi_2 dt dx\\
= (2\pi)^{-2} \pi \Lambda \int_{-1}^1 e^{-t} \int_{\xi_1^2 + \xi_2^2<1} d\xi_2 d\xi_1
= \frac14 \Lambda \left(e- e^{-1}\right)=L^{\rm cl}_{0,2}\Lambda
|\Omega|_h.
\end{multline*}

\begin{figure}[htb]
	\centerline{\psfig{file=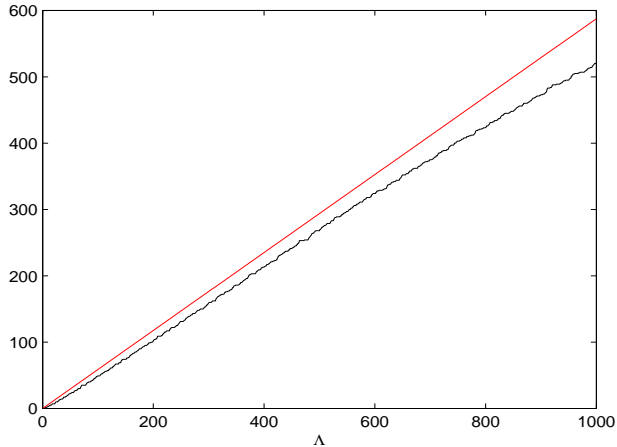,width=9.5cm,height=6.5cm,angle=0}}
	\caption{The graph $\mathcal N(\Lambda)$ shown in black and
the graph of $L^{\rm cl}_{0,2}\Lambda |\Omega|_h$ is shown in red. }
	\label{Fig.2}
\end{figure}

Let us say a few words about the calculations of the eigenvalues of the collection of
problems~\eqref{lk}. We use the Chebyshev differentiation matrix~\cite{Matlab}
for the spectral approximation
of the derivative (and this matrix squared   for the second derivative)
for the  numerical  solution of the  eigenvalue problems~\eqref{lk}
(we observe that the potentials are analytic).
We have used matrices of order $400\times 400$.
The accuracy is tested against the problem~\eqref{lk} with
$l=0$, so that the corresponding eigenvalues $n^2\cdot\frac4{\pi^2}$
are computed for $n=1,\dots,200$ with correct 14 decimal places.
We therefore reasonably expect that the accuracy is of the similar order
for $\ell\ge1$.

We set $\Lambda\le1000$. Then to calculate
$\mathcal N(\Lambda)$ it is enough to limit $l\le50$,
since the first eigenvalue of ~\eqref{lk} with $l=50$ is already greater than $1000$.
The eigenvalues $\nu_{\ell k}$ are of the order $k^2\cdot\frac4{\pi^2}$
and the therefore the length $[1:200]$ of each the sequence of eigenvalues
$\nu_{\ell k}$ taken into account for each fixed $\ell\le 50$ is also more than enough for
$\Lambda\le1000$.

\bigskip
\bigskip
\noindent

{\it Acknowledgements:} The authors would like to thank N. A. Zaitsev
for his help and advice concerning the computer calculations of the
eigenvalues of the Sturm--Liouville  problems~\eqref{lk}.

The work of A.I. was supported by
Moscow Center for Fundamental and Applied Mathematics, Agreement
with the Ministry of Science and Higher Education of the Russian
Federation, No.~075-15-2022-283. AL was supported by
the Ministry of Science and Higher
Education of the Russian Federation,
(Agreement 075-10-2021-093, Project MTH-RND-2124).

\end{document}